\documentclass[12pt]{amsart}
\usepackage{amssymb,amsthm,amsmath,amstext}
\usepackage{mathrsfs}  
\usepackage{mathtools} 
\usepackage{colonequals}

\usepackage{multirow}
\usepackage{fullpage}

\usepackage[all]{xy}

\usepackage[backref=page]{hyperref}
\hypersetup{pdftitle={Kneser's method of neighbors}}
\hypersetup{pdfauthor={John Voight}}
\hypersetup{colorlinks=true,linkcolor=blue,anchorcolor=blue,citecolor=blue}

\usepackage[nameinlink]{cleveref}

\numberwithin{equation}{section}

\theoremstyle{plain}

\newtheorem{thm}[equation]{Theorem}

\newtheorem{prop}[equation]{Proposition}
\newtheorem{lemma}[equation]{Lemma}

\newtheorem{cor}[equation]{Corollary}

\theoremstyle{definition}
\newtheorem{defn}[equation]{Definition}

\newtheorem{example}[equation]{Example}

\theoremstyle{remark}
\newtheorem{remark}[equation]{Remark}
\newtheorem{rmk}[equation]{Remark}

\newcommand{\Z}{\mathbb Z}

\newcommand{\C}{\mathbb C}
\newcommand{\F}{\mathbb F}
\renewcommand{\P}{\mathbb P}
\newcommand{\Q}{\mathbb Q}
\newcommand{\R}{\mathbb R}

\DeclareMathOperator{\Aut}{\mathrm{Aut}}

\DeclareMathOperator{\disc}{\mathrm{disc}}

\DeclareMathOperator{\GL}{\mathrm{GL}}
\DeclareMathOperator{\M}{\mathrm{M}}

\DeclareMathOperator{\Mat}{\mathrm{Mat}}

\DeclarePairedDelimiter\abs{\lvert}{\rvert}

\newcommand{\frakp}{\mathfrak{p}}

\newcommand{\linedef}[1]{\textsf{#1}}
\newcommand{\defi}[1]{\linedef{#1}}

\DeclareMathOperator{\Orth}{O}
\DeclareMathOperator{\SO}{SO}
\DeclareMathOperator{\Spin}{Spin}
\DeclareMathOperator{\diag}{diag}
\DeclareMathOperator{\Gen}{Gen}
\DeclareMathOperator{\Cl}{Cl}
\DeclareMathOperator{\Map}{Map}
\DeclareMathOperator{\impart}{Im}

\DeclareMathOperator{\tr}{tr}

\newenvironment{enumroman}
{\begin{enumerate}}
{\end{enumerate}}

\newcommand{\tprodprime}[1]{\textstyle{\prod^{\prime}_{#1}}}

\begin{document}

\title{Kneser's method of neighbors}

\author{John Voight}
\address{Department of Mathematics, Dartmouth College, Kemeny Hall, Hanover, NH 03755, USA}
\email{jvoight@gmail.com}

\date{\today}

\begin{abstract}
In a landmark paper published in 1957, Kneser \cite{Kneser} introduced a method for enumerating classes in the genus of a definite, integral quadratic form.  This method has been deeply influential, on account of its theoretical importance as well as its practicality.  In this survey, we exhibit Kneser's method of neighbors and indicate some of its applications in number theory.  
\end{abstract}

\maketitle

\section{Introduction}

The study of integral quadratic forms, a subject with classical roots, remains of enduring interest today.  A fundamental problem concerns the classification of forms up to equivalence under invertible change of variables.  A seminal contribution was made by Martin Kneser \cite{Kneser} in a short but impactful paper published in 1957.  In this paper, Kneser introduced the notion of a \emph{neighbor} of a quadratic form; he then used this notion to solve the classification problem for quadratic forms of small discriminant and moderate rank.  Since then, Kneser's neighboring method has been used in a number of different settings and has found many applications.  

In this survey, we begin in \cref{sec:setup} by exhibiting Kneser's method in its simplest form.  Then in \cref{sec:lattices}, we reinterpret this method in the language of lattices and orthogonal groups and then comment on algorithms and applications in \cref{sec:appl}.  Finally, we explain in \cref{sec:modularforms} how neighbors arise naturally from the consideration of Hecke operators acting on spaces of modular forms.

\subsection*{Acknowledgements}

The author would like to thank Eran Assaf, Wai Kiu (Billy) Chan, Dan Fretwell, Matthew Greenberg, Markus Kirschmer, Gabriele Nebe, and Rainer Schulze-Pillot for their informative perspectives on the topic; further thanks go to Simon Broadhurst, Ga\"etan Chenevier, Mathieu Dutour Sikiri\'{c}, Noam Elkies, Markus Kirschmer, Spencer Secord, and Haochen Wu for corrections and suggestions.  The author was supported by a Simons Collaboration grant (550029). 

\section{Neighbors, a first meeting} \label{sec:setup}

In this section, with a rapid setup we describe neighbors in a simplified situation to get the main idea across.  Let 
\begin{equation} 
Q(x)=Q(x_1,\dots,x_n)=\sum_{1 \leq i \leq j \leq n} a_{ij} x_ix_j \in \Z[x_1,\dots,x_n] 
\end{equation}
be an \defi{integral quadratic form}, a homogeneous polynomial of degree $2$ with $a_{ij} \in \Z$.  Associated to $Q$ is a $\Z$-bilinear form
\begin{equation} \label{eqn:Tzn}
\begin{aligned}
T \colon \Z^n \times \Z^n &\to \Z \\
T(x,y) &= Q(x+y)-Q(x)-Q(y)
\end{aligned} 
\end{equation}
with $Q(x)=\frac{1}{2}T(x,x)$ for all $x \in \Z^n$.  The \defi{Gram matrix} of $Q$ is the symmetric matrix 
\begin{equation}
[T] \colonequals [T(e_i,e_j)]_{i,j}=\begin{pmatrix}
2a_{11} & a_{12} & \dots & a_{1n} \\
a_{12} & 2a_{22} & \dots & a_{2n} \\
\vdots & \vdots & \ddots & \vdots \\
a_{1n} & a_{2n} & \dots & 2a_{nn}
\end{pmatrix} \in \M_n(\Z),
\end{equation}
where $e_1,\dots,e_n$ are the standard basis vectors for $\Z^n$.  Then $x^{\intercal} [T] y = T(x,y)$ for all $x,y \in \Z^n$, where ${}^{\intercal}$ denotes the transpose.  The \defi{(half-)discriminant} of $Q$ is 
\begin{equation} 
\disc Q \colonequals 2^{-\varepsilon(n)} \det [T] \in \Z_{>0} 
\end{equation}
where $\varepsilon(n)=0,1$ according as $n$ is even or odd.  

\begin{remark}
The apparent break in symmetry according to parity of the rank allows us to treat $p=2$ uniformly below (and reflects differences between odd and even rank quadratic forms).
\end{remark}

For the rest of this section, suppose that $Q$ is \defi{positive definite}, i.e., $Q(x) > 0$ for all nonzero $x \in \Z^n$.  Let $\mathcal{Q}_{n,d}$ be the set of positive definite (integral) quadratic forms in $n$ variables with discriminant $d$.  The group $\GL_n(\Z)$ acts by invertible change of variables on $\mathcal{Q}_{n,d}$ via $(gQ)(x)=Q(g^{-1}x)$.  We say $Q,Q' \in \mathcal{Q}_{n,d}$ are \defi{equivalent} if $Q'=gQ$ for some $g \in \GL_n(\Z)$, writing also $Q' \sim Q$.  In terms of Gram matrices, we have $Q \sim Q'$ if and only if there exists $A \in \GL_n(\Z)$ such that $[T']=A^{\intercal} [T] A$.  The set of \defi{automorphisms} (self-equivalences)
\begin{equation}  \label{eqn:AutQ}
\Aut Q \colonequals \{g \in \GL_n(\Z) : gQ = Q\} 
\end{equation}
of $Q$ is a finite group.

For $n=2$, following Gauss we say that $Q(x,y)=ax^2+bxy+cy^2$ is \defi{reduced} if $0 \leq b \leq a \leq c$; from these inequalities, a complete, finite set of representatives for a given discriminant $d=\disc Q=\abs{b^2-4ac}$ can be enumerated.  For general $n \geq 2$, the theory of \emph{Minkowski reduction} similarly provides a polyhedral domain, defined by inequalities on the coefficients $a_{ij}$, which can be used to prove the following theorem. 

\begin{thm} \label{thm:finiteness}
The set of $\GL_n(\Z)$-equivalence classes in $\mathcal{Q}_{n,d}$ is finite.
\end{thm}

\begin{proof}
See Cassels \cite[Chapter 12, Theorem 1.1]{Cas}; for further detailed references, see 
Conway--Sloane \cite[Chapter 15, \S 10]{CS}.
\end{proof}

One is left with the fundamental problem of enumerating representatives for the classes in $\mathcal{Q}_{n,d}$ and to compute its size $h_{n,d} \colonequals \#(\GL_n(\Z) \backslash \mathcal{Q}_{n,d})$.  This is where Kneser begins his article \cite{Kneser}.  The class number $h_{n,d}$ grows rapidly with $n$ and $d$, so we are often restricted to small values.  Unfortunately, using Minkowski reduction quickly becomes impractical already for $n \geq 6$, due to the complicated description of the domain and its boundary.  On the other hand, the Smith--Siegel--Minkowski mass formula \cite[section 4]{CSIV} provides an explicit expression for a \emph{weighted} class number
\begin{equation} \label{eqn:yuyub}
m_{n,d} \colonequals \sum_{[Q] \in \mathcal{Q}_{n,d}/\sim} \frac{1}{\#\Aut Q};
\end{equation}
for example, if $d=1$ and $8 \mid n$ then
\[ m_{n,d} = \frac{\abs{B_{n/2}}}{n} \prod_{k=1}^{(n/2)-1} \frac{\abs{B_{2k}}}{4k} \]
where $B_k$ are the Bernoulli numbers defined by $x/(e^x-1)=\sum_{k=0}^{\infty} B_k x^k/k!$.
Therefore, if one has a list $Q_1,\dots,Q_h$ of inequivalent forms in $\mathcal{Q}_{n,d}$, then one can certify the list is complete by showing that $m_{n,d}=\sum_i 1/(\#\Aut Q_i)$.  However, the mass formula does not itself provide a set of representatives, and it is not clear in general how to compute $h_{n,d}$ from $m_{n,d}$.

Noting these two strategies and their limitations, Kneser \cite{Kneser} proposed a different method, starting with a quadratic form $Q \in \mathcal{Q}_{n,d}$.  Intuitively, we will look for an \emph{adjacent} form $Q' \in \mathcal{Q}_{n,d}$ obtained from $Q$ by a change of variables with denominator $2$.  We illustrate the method with an example.

\begin{example} \label{exm:311}
We begin with the quadratic form
\[ Q(x,y,z) = x^2+y^2+yz+3z^2  \]
with Gram matrix
\[ [T] = \begin{pmatrix} 2 & 0 & 0 \\ 0 & 2 & 1 \\ 0 & 1 & 6 \end{pmatrix} \]
and $\disc Q = 11$, so $Q \in \mathcal{Q}_{3,11}$.  We look for a vector we can divide by $2$ (retaining integral values) and we find $Q(1,0,1)=4$.  With a change of variables, we can make this the first basis vector, giving the equivalent form 
\begin{equation}  \label{eqn:46yz}
4x^2 + xy + 6xz + y^2 + yz + 3z^2. 
\end{equation}
Substituting $(\tfrac{1}{2} x, 2y, z)$ in \eqref{eqn:46yz} gives
\begin{equation} 
x^2 + xy + 3xz + 4y^2 + 2yz + 3z^2 \sim Q'(x,y,z) \colonequals x^2+xy+xz+y^2+yz+4z^2. 
\end{equation}
We have $\disc Q' = 11$, so $Q' \in \mathcal{Q}_{3,11}$.  We claim that $Q' \not\sim Q$: 
indeed $Q(x,y,z)=1$ has only four solutions but $Q'(x,y,z)=1$ has six.  

The mass formula for $m_{3,p}$ for prime $p$ is $(p-1)/48$.  For $p=11$ we get $m_{3,11}=5/24$.  We compute that 
\begin{equation}
\Aut(Q)=\langle -1, \left(\begin{smallmatrix} 1 & 0 & 0 \\ 0 & 1 & 0 \\ 0 & 1 & -1 \end{smallmatrix}\right), \left(\begin{smallmatrix} -1 & 0 & 0 \\ 0 & 1 & 0 \\ 0 & 0 & 1 \end{smallmatrix}\right) \rangle \simeq (\Z/2\Z)^2
\end{equation}
so $\#\Aut(Q)=8$, and similarly $\#\Aut(Q')=12$.  Since $5/24=1/8+1/12$, we conclude that there are exactly $2$ ternary forms of discriminant $11$ up to equivalence, represented by $Q$ and $Q'$.
\end{example}

Via this method, Kneser \cite[Satz 3]{Kneser} computes representatives for quadratic forms in $n \leq 16$ variables and small discriminant $d$ with just a few pages of calculation.  In particular, with clever but short arguments, he quickly recovers two statements:
\begin{itemize}
\item a theorem of Mordell \cite{Mordell} that every form in $\mathcal{Q}_{8,1}$ (in $8$ variables with discriminant $1$) is equivalent to 
\begin{equation} \label{eqn:E8}
\left(\sum_{1 \leq i \leq j \leq 8} x_ix_j\right) - x_1x_2 - x_2x_3,
\end{equation}
and 
\item a theorem of Witt \cite[Satz 3]{Witt} that there are exactly two classes in $\mathcal{Q}_{16,1}$.
\end{itemize}

\begin{remark}
The two classes in $\mathcal{Q}_{16,1}$ were famously used by Milnor \cite{Milnor} to exhibit isospectral, nonisometric flat tori.
\end{remark}

\section{Neighboring lattices} \label{sec:lattices}
  
With the basic idea explained in the previous section, we now shift our perspective slightly to give a more general treatment: thinking geometrically, we consider lattices in a quadratic space.  As general references, see Serre \cite[Chapters IV--V]{Serre}, O'Meara \cite{OM}, Conway--Sloane \cite[Chapter 15]{CS}, and Chenevier--Lannes \cite[Chapters 2--3]{CL}.

\subsection*{Notation}

Let $V$ be a $\Q$-vector space with $n \colonequals \dim_\Q V$.  Let $Q \colon V \to \Q$ be a nondegenerate quadratic form with associated bilinear form $T \colon V \times V \to \Q$ (as in \eqref{eqn:Tzn}).  We say $Q$ is \defi{isotropic} if $Q(v)=0$ for some $v \neq 0$.  

A \defi{(full) lattice} $\Lambda \subset V$ is the $\Z$-span of a $\Q$-basis for $V$.  A lattice $\Lambda$ is \defi{integral} if $Q(\Lambda) \subseteq \Z$; this may always be obtained by rescaling.

\begin{rmk}
Since $T(x,x)=2Q(x)$ for all $x \in \Lambda$, we work implicitly throughout with what are called \defi{even} lattices.  The definitions extend as well to \defi{odd} lattices, those for which $T(\Lambda,\Lambda) \subseteq \Z$ but $Q(\Lambda) \not\subseteq \Z$, at the expense of additional technical conditions for $p=2$.
\end{rmk}

Let $\Lambda \subset V$ be an integral lattice and choose a basis $\Lambda=\Z e_1 + \dots + \Z e_n$.  Then we obtain an integral quadratic form
\[ Q_\Lambda(x) = Q(x_1e_1 + \dots + x_ne_n) \in \Z[x_1,\dots,x_n]   \]
as in the previous section.  Conversely, given a quadratic form $Q$ on $\Z^n$, extending scalars we have the integral lattice $\Lambda = \Z^n \subset V \colonequals \Q^n$.  As before, we define the \defi{discriminant}
\begin{equation} 
\disc \Lambda \colonequals 2^{-\varepsilon(n)} \det(T(e_i,e_j))_{i,j} \in \Z_{>0} 
\end{equation}
(well-defined, independent of the choice of basis).  We note that for a sublattice $\Lambda' \subseteq \Lambda$, 
\begin{equation} \label{eqn:dlambda}
\disc \Lambda' = [\Lambda:\Lambda']^2 \disc \Lambda.
\end{equation}
We will abbreviate $d \colonequals \disc \Lambda$.

\begin{example} \label{exm:unimodular}
Lattice of discriminant $1$ are called \defi{unimodular} or \defi{self-dual}.  (For further reading, see Chenevier--Lannes \cite[Chapter 2]{CL} or Conway--Sloane \cite{CS4}.)  There exists a unimodular lattice in even dimension $n$ if and only if $n \equiv 0 \pmod{8}$.  Associated to a unimodular lattice is its root system $\{x \in \Lambda : Q(x)=1\}$ of rank at most $n$; its irreducible components are of ADE type ($\mathbf{A}_n$, $\mathbf{D}_n$, $\mathbf{E}_6$, $\mathbf{E}_7$, or $\mathbf{E}_8$).

For example, for $n \geq 3$ we consider the quadratic space $\Q^n$ with quadratic form $Q(x)=\tfrac{1}{2}\sum_{i=1}^{n} x_i^2$.  Inside we have the integral lattice $D_n \colonequals \{x \in \Z^n : \sum_{i=1}^n x_i \equiv 0 \pmod{2}\}$ with root lattice of type $\mathbf{D}_n$.  For $8 \mid n$, the lattice $D_n^+ \colonequals D_n + \Z v$ where $v = \tfrac{1}{2}(1,\dots,1)$ is a unimodular lattice.  The root lattice of $D_8^+$ is of type $\mathbf{E}_8$ (as is the $\Z$-span $E_8$ of $\mathbf{E}_8$, giving the exceptional identity $E_8=D_8^+$), and its associated integral quadratic form is \eqref{eqn:E8}.  The root lattice of $D_{16}^+$ is of type $\mathbf{D}_{16}$ and the two classes in $\mathcal{Q}_{16,1}$ mentioned above are represented by $D_{16}^+$ and $E_8 \boxplus E_8$.  
\end{example}

We define the \defi{orthogonal group} of $V$
\begin{equation}
\Orth(V) \colonequals \{g \in \GL(V) \colon Q(gx) = Q(x) \text{ for all $x \in V$}\}
\end{equation}
as the group of invertible linear maps which preserve the quadratic form $Q$.  We call the elements of $\Orth(V)$ \defi{isometries} of $V$.  We also consider the subgroup of isometries that preserve $\Lambda$, namely
\begin{equation}
\Orth(\Lambda) \colonequals \{g \in \Orth(V) : g \Lambda = \Lambda\}.
\end{equation}
When $Q$ is positive definite, we have $\Orth(\Lambda) \simeq \Aut Q_{\Lambda}$, as in \eqref{eqn:AutQ}; in particular, $\#\Orth(\Lambda) < \infty$.

We say that lattices $\Lambda,\Pi \subset V$ are \defi{isometric}, written $\Lambda \simeq \Pi$, if there exists $g \in \Orth(V)$ such that $g \Lambda = \Pi$.  Choosing bases for $\Lambda,\Pi$, we see that $\Lambda \simeq \Pi$ if and only if $Q_\Lambda \sim Q_{\Pi}$.

We seek to classify isometry classes of lattices, and a first step is given by considering their invariants locally.  More precisely, for a prime $p$, let $\Q_p$ be the field of $p$-adic numbers and $\Z_p \subset \Q_p$ the ring of $p$-adic integers.  We repeat the above definitions, replacing $\Z,\Q$ with $\Z_p,\Q_p$.  To $\Lambda \subset V$, we associate $\Lambda_p \colonequals \Lambda \otimes_{\Z} \Z_p \subseteq V_p \colonequals V \otimes_{\Q} \Q_p$.  Detecting local isometry is much easier than global isometry; for example, we have the following lemma.  

\begin{lemma} \label{lem:LambdapV}
If $\Lambda' \subset V$ is an integral lattice with $d=\disc \Lambda=\disc \Lambda'$, then $\Lambda_p' \simeq \Lambda_p$ for all $p \nmid d$.  
\end{lemma}

\begin{proof}
See Conway--Sloane \cite[Chapter 15, section 7]{CS}.
\end{proof}

We define the \defi{genus} of $\Lambda$ to be
\begin{equation} 
\Gen \Lambda \colonequals \{ \Lambda' \subset V : \Lambda_p \simeq \Lambda'_p \text{ for all $p$}\}. 
\end{equation}
(There is no condition over $\R$ since the lattices are taken in $V$, which has a fixed signature.)  For all $\Lambda' \in \Gen \Lambda$ we have $\disc \Lambda' = \disc \Lambda$, so the set $\mathcal{Q}_{n,d}$ is a disjoint union of (finitely many) genera.  The \defi{class set} 
\begin{equation}
\Cl \Lambda \colonequals (\Gen \Lambda)/\!\simeq
\end{equation} 
is the set of (global) isometry classes in the genus $\Gen \Lambda$.  The genus of $\Lambda$ consists of all lattices in $V$ which are (everywhere) locally isometric to $\Lambda$.  Therefore, the class set measures the failure of the corresponding local--global principle: namely, when local isometry classes determine the global isometry class.  Again by reduction theory (\Cref{thm:finiteness})---extended to the indefinite case---the class set is finite: $\#\Cl \Lambda <\infty$.  

\subsection*{Neighbors}

Kneser's theory of $p$-neighbors \cite{Kneser} gives an effective method to enumerate representatives of the isometry classes in $\Cl \Lambda$.  Let $p \nmid d=\disc(\Lambda)$ be prime.

\begin{defn}
A lattice $\Pi \subset V$ is a \defi{$p$-neighbor} of $\Lambda$, written $\Pi \sim_p \Lambda$, if $\Pi$ is integral and
\[ [\Lambda:\Lambda \cap \Pi]=[\Pi:\Lambda \cap \Pi]=p. \]
\end{defn}

The $p$-neighbor relation is symmetric, but not reflexive or transitive; we should think of two $p$-neighbors as being \emph{adjacent}, differing only by a minimal index $p$.  (See below for a further combinatorial interpretation.)  We might visualize the $p$-neighbor relation and lattice containments like this:
\begin{equation}
\begin{aligned}
\xymatrix{
\Lambda \ar@{-}[dr]^{p} & & \Pi \ar@{-}[dl]_{p} \\
& \Lambda \cap \Pi \\
p\Lambda \ar@{-}[ur] & & p\Pi \ar@{-}[ul]}
\end{aligned}
\end{equation}

\begin{lemma}
If $\Lambda \sim_p \Pi$, then $\disc \Pi=\disc \Lambda$ and $\Pi \in \Gen \Lambda$.
\end{lemma}

\begin{proof}
We have $\disc \Pi=(1/p)^2 \disc (\Lambda \cap \Pi) = \disc \Lambda$, using \eqref{eqn:dlambda}.  Moreover, $\Pi_\ell=\Lambda_\ell$ for all primes $\ell \neq p$, and $\Pi_p \simeq \Lambda_p$ by \Cref{lem:LambdapV} (since $p \nmid d$), so $\Pi \in \Gen \Lambda$.
\end{proof}

Happily, $p$-neighbors admit several further explicit characterizations, as follows.  Let $H\colon \Z^2 \to \Z$ denote the \defi{hyperbolic plane}, defined by $H(x,y)=xy$.  We write $\Lambda \boxplus \Lambda'$ for the orthogonal direct sum of the lattices $\Lambda,\Lambda'$ (inside the orthogonal direct sum $V \boxplus V'$).

\begin{prop} \label{laticpr}
Let $\Pi \subset V$ be a lattice.  Then the following are equivalent.
\begin{enumroman}
\item $\Pi \sim_p \Lambda$, i.e., $\Pi$ is a $p$-neighbor of $\Lambda$.
\item $\Lambda_\ell=\Pi_\ell$ for all primes $\ell \neq p$, and there exists a splitting 
\[ \Lambda_p = (\Z_p e_1 \oplus \Z_p e_2) \boxplus \Lambda_p'   = H_p \boxplus \Lambda_p'   \]
such that
\[ \Pi_p = \Z_p \Bigl({\frac{1}{p}}e_1\Bigr) \oplus \Z_p(pe_2) \boxplus \Lambda_p'.   \]
\item There exists $v \in \Lambda \smallsetminus p\Lambda$ such that $p^2 \mid Q(v)$ and
\[ \Pi = \Lambda(p,v) \colonequals (p^{-1}v)\Z + \{w \in \Lambda : T(v,w) \in p\Z\}.   \]
\end{enumroman}
\end{prop}

\begin{proof}
See Schulze-Pillot \cite[\S 1]{SP}, or more generally Scharlau--Hemkemeier \cite[\S 2]{SchHem} or Green\-berg--Voight \cite[\S 5]{GV}.
\end{proof}

\begin{remark}
Statement (ii) can be rephrased as saying that the \emph{invariant factors} of $\Pi$ with respect to $\Lambda$ are $1/p,p,1,\dots,1$.  Another interpretation may be given in terms of \emph{transverse Lagrangians}.  In particular, we recover the notion of neighbors from the previous section, where we took $p=2$.
\end{remark}

\begin{cor} \label{cor:pneigh}
Let $V(Q)$ denote the quadric defined by $Q(x)=0$ in $\P^{n-1}(F_p)$.  Then the set of $p$-neighbors is in (well-defined) bijection with the set of $\F_p$-points $V(Q)(\F_p) \subseteq \P^{n-1}(\F_p)$ via
\[ \Pi=\Lambda(p,v) \leftrightarrow [v]. \]
\end{cor}

In other words, the $p$-neighbors of $\Lambda$ correspond to isotropic lines in $\Lambda/p\Lambda$, hence their number is $\sim p^{n-2}$.  For example, if $n=3$ then the ternary quadratic form $Q$ defines a smooth conic in $\P^2$, so $Q_{\F_p} \simeq \P_{\F_p}^1$ and $\#Q(\F_p)=p+1$.  If $\disc Q \in \F_{p}^{\times 2}$ and $n$ is even, then 
\begin{equation} \label{eqn:Qnum}
\#Q(\F_p)=1+p+\dots+p^{n/2-2}+2p^{n/2-1}+p^{n/2}+\dots+p^{n-2}. 
\end{equation}

\begin{proof}
We check by definition in \Cref{laticpr}(iii) that $\Lambda(p,v)$ only depends on the $\F_p$-line spanned by $\overline{v} \in \Lambda/p\Lambda \simeq \Lambda \otimes \F_p$, so the map is well-defined.  To see that the map is bijective, we use Hensel's lemma (which applies even when $p=2$): for each $\overline{v} \in \Lambda/p\Lambda$ such that $Q(\overline{v})=0$, there exists $v \in \Lambda$ reducing to $\overline{v}$ such that $p^2 \mid Q(v)$.  
\end{proof}

\begin{example} \label{exm:thompson}
The vector $v \colonequals (0,1,2,\dots,23) \in D_{24}^+$ has $Q(v) = \sum_{i=1}^{23} i^2 = 4324 \equiv 0 \pmod{47}$.  The $47$-neighbor $D_{24}^+(47,v)$ (replacing $v$ with a Hensel lift) is isometric to the famed Leech lattice, whose automorphism group is the Conway group $\mathrm{Co}_0$ (which modulo its center of order $2$ is the Conway group $\mathrm{Co}_1$, a simple group).  This construction is attributed to Thompson \cite[Preface to Third Edition, (14), p.~lvi]{CS0}; see \cref{sec:appl} for further discussion.
\end{example}

\subsection*{Indefinite lattices and iterated neighbors}

Throughout this section, we suppose that $n \geq 3$ (the case $n=2$ being treated using reduction theory following Gauss, and the arithmetic of quadratic fields).  

Kneser's notion of $p$-neighbor gives us a way to exhibit lattices in the genus of $\Lambda$.  As Kneser explains \cite[p.~242]{Kneser}, he was led to this notion by consideration of the easier case when $Q$ is indefinite (equivalently, $Q$ is isotropic over $\R$) \cite{Kneser5}.  More precisely, when the spinor genus coincides with the genus \cite[Chapter 15, \S 9]{CS}, then $\#\Cl \Lambda = 1$ by a result of Eichler \cite{Eichler}.  Indeed, for $p \nmid \disc(\Lambda)$, Eichler calls the $\Z[1/p]$-lattice $\Lambda[1/p] \colonequals \Lambda \otimes \Z[1/p]$ \defi{arithmetically indefinite}, since $Q$ becomes isotropic over $\Q_p$.  Again when the spinor genus coincides with the genus, he concludes that $\#\Cl(\Lambda \otimes \Z[1/p])=1$.  

As a consequence, Kneser observes that if $\Lambda' \in \Gen \Lambda$, then there exists $g \in \Orth(V)$ with $g \in \GL_n(\Z[1/p])$ such that $g\Lambda' \simeq \Lambda$.  Without loss of generality, we may consider the case where $g\Lambda' = \Lambda$.  Then $[\Lambda : \Lambda \cap \Lambda']=[\Lambda' : \Lambda \cap \Lambda']$ is a power of $p$.  Inductive considerations then show the following general theorem.  

\begin{thm} \label{thm:sa}
There exists a finite set $S$ of primes $p$ (depending on $\Gen \Lambda$) with $p \nmid d$ such that every $\Lambda' \in \Gen \Lambda$ is connected to $\Lambda$ by a chain of neighbors.  More precisely, there exist primes $p_1,\dots,p_r \in S$ (not necessarily distinct) and lattices $\Pi_1,\dots,\Pi_r \in \Gen \Lambda$ such that
\begin{equation}
\Lambda \sim_{p_1} \Pi_1 \sim_{p_2} \cdots \sim_{p_r} \Pi_r \simeq \Lambda'. 
\end{equation}
Moreover, if $n \geq 3$ and $d$ is squarefree, then we may take $S=\{p\}$ for any prime $p \nmid d$.
\end{thm}

\begin{proof}
For the unimodular case, see Chenevier--Lannes \cite[Theorem 3.1.12]{CL} for a proof in the spirit of Kneser's original argument.  More generally, see Greenberg--Voight \cite[Theorem 5.8]{GV} and Scharlau--Hemkemeier \cite[end of \S 2]{SchHem}.  
\end{proof}

\begin{rmk}
In \Cref{thm:sa}, concerning the second statement, Hsia--Jochner \cite[Corollary 4.1]{HJ} have shown that when $d$ is squarefree and $p$ sufficiently large, in fact every class in $\Gen \Lambda$ is directly a $p$-neighbor of $\Lambda$, without iteration.  For the final statement, the given hypothesis on the discriminant may be replaced by the (more general) hypothesis that the spinor genus coincides with the genus.
\end{rmk}

\begin{example} \label{exm:nomass}
We return to \Cref{exm:311}, taking $\Lambda=\Z^3=\Lambda_1$.  The solutions to $Q(x,y,z) \equiv 0 \pmod{2}$ are given by
\[ (1:0:1),(1:1:0),(1:1:1) \]
giving rise to $2+1=3$ neighbors for $p=2$.  The first vector lifts to $v=e_1+e_3$ as in the previous example, giving rise to the lattice 
\begin{equation}
\begin{aligned} 
\Lambda_2 &=\Lambda(2,v)=\frac{1}{2}(e_1+e_3)\Z + e_1 \Z + (2e_2)\Z + e_3 \Z \\
&= \frac{1}{2}(e_1+e_3)\Z + (2e_2)\Z + e_3 \Z
\end{aligned}
\end{equation}
with $\Lambda_2 \not\simeq \Lambda_1$.  For the $2$ other neighbors, we lift and find that $\Lambda(2,(1,1,2)) \simeq \Lambda_1$ and $\Lambda(2,(1,-1,1)) \simeq \Lambda_2$, classes we have already seen. 

Repeating with $\Lambda_2$, we find that all of its $2$-neighbors are isometric to $\Lambda_1$.  And we will not see any new classes from the other $2$-neighbors!  We conclude again that $\#\Cl \Lambda = 2$.
\end{example}

In particular, \Cref{exm:nomass} shows that even without using the mass formula, one can exhaust representatives for $\Cl \Lambda$ using \Cref{thm:sa}.  

\begin{remark}
\Cref{thm:sa} fits into the broader framework of \emph{strong approximation}, a subject also pioneered by Kneser \cite{Kneser-strong}: for a complete history and further reference, see the survey by Scharlau \cite{Scha2009}.  Let $F$ be a global field with $\underline{F} \colonequals \tprodprime{v} F_v$ the adele ring of $F$.  Let $S$ be a finite set of places of $F$.  Let $G$ be a simply-connected, simple linear algebraic group defined over $F$, such that $G(F_S) \colonequals \prod_{v \in S} G(F_v)$ is not compact.  Then $G$ has the \emph{strong approximation property} with respect to $S$: namely, $G(F)G(F_S)$ is dense in $G(\underline{F})$.  

The complications in \Cref{thm:sa} above arise because $\Orth(V)$ is not connected and $\SO(V)$ is not simply connected, but the spin cover $\Spin(V)$ of $\SO(V)$ \emph{is} simply connected.  Over $F=\Q$, taking $S=\{\infty\}$ we see that if $V$ is indefinite, then the spinor genus and the class agree; if $V$ is definite, but $\Orth(V_p)$ is indefinite (i.e., $V_p$ is isotropic), then every class in the spinor genus of $\Lambda$ has a representative which differs only from $\Lambda$ at $p$.  
\end{remark}

\section{Algorithms, generalizations, and applications} \label{sec:appl}

In this section, we survey some further practical and theoretical aspects and applications of neighbors.
 
\subsection*{Algorithms}

The method of $p$-neighbors can be readily implemented on a computer.  For an algorithmic description, see Schulze-Pillot \cite[\S 2]{SP} or Greenberg--Voight \cite[\S 3]{GV}.  Implementations are available in \textsf{Magma} \cite{Magma}, \textsf{Sage} \cite{Sage}, as well as other specialized packages for work with lattices.  The set of $p$-neighbors can be computed in time $O(p^{n-2+\epsilon} H_n(\|\Lambda\|))$ for any $\epsilon>0$, where $\|\Lambda\|$ is the bit size of the Gram matrix and $H_n$ is a polynomial depending only on $n$, counting bit operations in computing the Hermite normal form of an integer matrix: for more detail, see Hein \cite[\S 5.4]{Hein}.

The most time-consuming step in the above strategy of iterating $p$-neighbors concerns isometry testing.  When the quadratic space $V$ is \emph{definite} (the interesting case), there is a practical algorithm to check if two lattices $\Lambda,\Lambda'$ are isometric due to Plesken--Souvignier \cite{PS}: they match up short vectors, using many pruning shortcuts and other optimizations to either exhibit an isometry or rule it out as early as possible.  For theoretical purposes, Haviv--Regev \cite{HR} proposed algorithms for this purpose with a complexity of $n^{O(n)} (\|\Lambda\|+\|\Lambda'\|)^{O(1)}$.  Indeed, the mass formula shows that the growth rate of the class number is at least $n^{n^2}$.

For fixed rank $n$, a deterministic, polynomial time, practical algorithm was exhibited by Dutour Sikiri\'{c}--Haensch--Voight--van Woerden \cite{canonical}, using reduction to a canonical form.  On the other hand, when the rank is small ($n \leq 5$), an appropriately optimized version of Minkowski reduction can be used: for example, when $n=3$ we can use Eisenstein reduction, and Schulze--Pillot \cite[pp.\ 138--140]{SP} works out the case $n=4$.  These methods also permit the computation of the finite group $\Orth(\Lambda)$.

These two estimates can be combined (depending on the algorithm used) to estimate the running time of computing representatives for the class set $\Cl \Lambda$, when the spinor genus is equal to the genus.  More generally, we need to know that the set $S$ in \Cref{thm:sa} is effectively computable, which can be shown using methods of class field theory.

However, the procedure becomes impractical for large rank $n$ because of the (exponentially) large number of neighbors.  On the other hand, if $\Lambda$ has a large automorphism group, one need only run through a set of representatives of $p$-neighbors under this group---this observation was already made by Kneser, and it was essential for his application.  (This can only improve the running time by a constant factor in fixed rank.)

\subsection*{Generalizations}

The notion of $p$-neighbors admits important generalizations, as follows.  

\begin{itemize}
\item For $k \geq 1$, we define a \defi{$p^k$-neighbor} of $\Lambda$ to be an integral lattice $\Pi \subset V$ such that $\Lambda/(\Lambda \cap \Pi) \simeq \Pi/(\Lambda \cap \Pi) \simeq (\Z/p\Z)^k$.  Since a $p^k$-neighbor is obtained from an iterated $p$-neighbor, this notion does not give a new approach to enumerating classes in the genus; it is, however, important in the application to modular forms (\cref{sec:modularforms}).  One may similarly define \emph{$A$-neighbors} for an arbitrary finite group $A$.
\item Using \Cref{laticpr}(ii), we can also adapt the notion of $p$-neighbors (that stay in $\Gen \Lambda$) when $p \mid d$ \cite{SP}.
\item There is also a natural extension of the notion of neighbors to Hermitian forms, relative to a  quadratic field $K$, as developed by Hoffmann \cite{Hoffmann} and Schiemann \cite{Schiemann}.  (For further generalizations to semisimple groups, including symplectic groups, see the final section.)  The method is substantially similar, including the application of strong approximation.  However, there are certain additional technicalities: owing, for example, to a nontrivial class group in the ring of integers of $K$, cases depending on the splitting behavior of $p$ in $K$, and so on.  
\item For both orthogonal and unitary groups, neighbors generalize to the case of (totally positive definite) lattices over the ring of integers of a totally real field.  Indeed, Greenberg--Voight \cite{GV} present neighbors in both cases in this level of generality.
\end{itemize}
   
\subsection*{Applications} 

Kneser's original application to the enumeration of classes in certain genera was further developed and extended by many authors.  

Perhaps the first most significant stride was made by Niemeier \cite{Niemeier} in 1968, who classified (even) lattices of discriminant $d=1$ in $n=24$ variables, finding exactly $24$ classes (accordingly called \defi{Niemeier lattices}) including the Leech lattice.  Although the method of neighbors shows that the class set in this case is effectively computable, a naive effort would involve testing far too many neighbors.  Instead, Niemeier studied certain sublattices, and from there shows uniqueness of neighbors without minimal vectors (those with $Q(x)=1$), and from there he shows that starting from the list of $24$, no new lattices are obtained from $2$-neighbors.  For more, see Conway--Sloane \cite{cs16} and Venkov \cite{Venkov-cs}.  

The classification of unimodular lattices using Kneser's method continues.  Venkov \cite{Venkov} used neighbors to obtain a partial classification of odd unimodular lattices of rank at most $24$.  Bacher \cite{Bacher} gave explicit constructions for the odd unimodular lattices of dimension at most $24$ as $\Z/k\Z$-neighbors of the lattice $\Z^n$ (with the dot product) with $k \in \Z_{\geq 2}$, in the spirit of \Cref{exm:thompson}; he also gave an analogous description for even unimodular lattices.  
Results also extend beyond rank $24$: see Bacher--Venkov \cite{BV} and upcoming work by Chenevier \cite{Chen-uni} and Allombert--Chenevier \cite{AC}, reaching rank $28$.

Scharlau--Hemkemeier \cite{SchHem} describe a computer implementation of the neighbor method over totally real fields; using their optimized implementation of $2$-neighbors over $\Z$, they classified \emph{$\ell$-elementary} lattices (those whose discriminant group is an elementary abelian $\ell$-group for a prime $\ell$) for small values of $\ell$ and $n$.

Neighbors also permit classification of extremal lattices.  In a nutshell, a lattice $\Lambda$ is \emph{extremal} if $\min \Lambda \colonequals \{ Q(x) : x \in \Lambda, x \neq 0\}$ is as large as possible (from the point of view of modular forms).  Scharlau--Schulze-Pillot \cite{SSP} give a self-contained overview of the topic and classify extremal lattices in moderate dimension using a computer implementation of Kneser's neighbor method.  (See also more recent work of Scharlau \cite{Scharlau-extremal}, settling an open case of dimension $14$ and level $7$ in which the genus has $83006$ classes!)

For the unitary case, Hoffmann \cite{Hoffmann} and Schiemann \cite{Schiemann} computed tables of class numbers of positive definite unimodular Hermitian forms over the rings of integers of certain imaginary quadratic fields using implementations of the (adapted) neighboring method.

Neighbors also are used in applications to algebraic geometry, where lattices arise naturally as Picard groups of surfaces equipped with their intersection form \cite{Nikulin}.  For example, neighbors were used to classify even hyperbolic lattices with infinitely many simple $(-2)$-roots which admit an isotropic vector with bounded inner product with all the simple $(-2)$-roots, giving a classification of K3 surfaces with zero entropy \cite{BM,X}.  Neighbors are also used in the classification symplectic birational involutions of some compact hyper-K\"ahler manifolds \cite{MM}.

Returning in some sense to the original application, we conclude by observing that neighbors can be used to determine isometries, as follows.  Two lattices $\Lambda_1$ and $\Lambda_2$ are in the same genus if and only if $H \boxplus \Lambda_1 \simeq H \boxplus \Lambda_2$, where $H$ is the hyperbolic plane (a statement likely known to Kneser).  Indeed, a constructive proof of this statement is given by Brandhorst--Elkies \cite[Lemma 2.6]{BE}, where an explicit isometry is computed using a neighbor relation between $\Lambda_1$ and $\Lambda_2$.  This approach was a key ingredient in computing equations for Hilbert modular surfaces using K3 surfaces by Elkies--Kumar \cite{EK}.

\section{Modular forms} \label{sec:modularforms}

In this section, we explain how $p$-neighbors define a notion of adjacency on the class set, and thereby a regular directed graph; the adjacency matrix of this graph defines a Hecke operator on a space of orthogonal modular forms.  This reveals a harmonizing, combinatorial structure underlying neighbors.  
Throughout, let $\Lambda \subset V$ be a lattice with discriminant $d=\disc \Lambda$ in a positive definite quadratic space $V$. 

\subsection*{Adjacency}

To exhaust representatives for $\Cl \Lambda$, we used iterated $p$-neighbors.  We now encode the number of neighbors we see (up to isometry) in a graph, as follows.

For each $p \nmid d$, we define the \defi{$p$-neighbor graph} $G_p=G_p(\Cl \Lambda)$ as follows.
\begin{itemize}
\item The set of vertices of $G_p$ is $\Cl \Lambda=\{[\Lambda_1],\dots,[\Lambda_h]\}$, where $h=\#\Cl \Lambda$.
\item For each $\Lambda_i$, and each $p$-neighbor $\Pi_{ij} \sim_p \Lambda_i$, we draw a directed edge $[\Lambda_i] \to [\Pi_{ij}]$.
\end{itemize}
We are now interested in not just the vertices of the graph, but also the edges (which are independent of the choice of representatives $\Lambda_i$).  The graph $G_p$ is a $k$-regular directed graph, where $k$ is the number of $p$-neighbors.  (In general, $G_p$ may have several connected components, due to spinor genera; see \Cref{thm:sa}.)

Let $[T_p]$ be the adjacency matrix of $G_p$.

\begin{example} \label{exm:collectingt2}
We recall \Cref{exm:nomass}, where we looked at $2$-neighbors among ternary quadratic lattices of discriminant $11$.  Collecting the neighboring relation in the graph $G_2$, we have
\begin{center}
\includegraphics{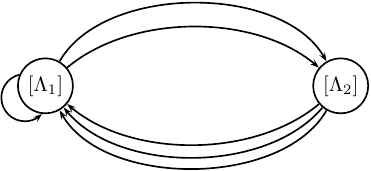}
\end{center}
and $[T_2]=\begin{pmatrix} 1 & 2 \\ 3 & 0 \end{pmatrix}$.  We similarly compute
$[T_3]=\begin{pmatrix} 2 & 2 \\ 3 & 1 \end{pmatrix}$, $[T_5]=\begin{pmatrix} 4 & 2 \\ 3 & 3 \end{pmatrix}$, and so on.
\end{example}

The $p$-neighbor graphs have remarkable properties.  For $n=3$, they are special types of expander graphs called \emph{Ramanujan graphs}: for $k\colonequals p+1$, they are $k$-regular (directed) graphs whose eigenvalues other than $\pm k$ have absolute value at most $2\sqrt{k-1}$.  For more, see Lubotzky--Phillips--Sarnak \cite{LPS} and Margulis \cite{M}.  For $n>3$, the $p$-neighbor graphs are no longer Ramanujan, but they still satisfy a spectral gap property.  In general, the diameters of these graphs are small, and so for large primes the neighbors of $\Lambda$ are equidistributed (weighted inversely proportional to automorphism groups) among the classes: see Schulze-Pillot \cite[Corollary p.~120]{SP86} and more generally recent work of Chenevier \cite[Theorem A]{CStat}.

\subsection*{Orthogonal modular forms and Hecke operators}

With this visual intuition, we now reinterpret the adjacency matrix as an operator on the space of functions on $\Cl \Lambda$.  The space of \defi{orthogonal modular forms for $\Lambda$} (with trivial weight) is
\[ M(\Orth(\Lambda))  \colonequals \Map(\Cl \Lambda,\C).   \]
In the basis of characteristic functions for $\Cl \Lambda$, we have $M(\Orth(\Lambda)) \simeq \C^h$ where $h \colonequals \#\Cl(\Lambda)$.  

For $p \nmid \disc(\Lambda)$, we define the \defi{Hecke operator} 
\begin{equation} 
\begin{aligned}
T_p \colon M(\Orth(\Lambda)) &\to M(\Orth(\Lambda))   \\
f &\mapsto T_p(f) \\  
T_p(f)([\Lambda'])   &= \sum_{\Pi' \sim_p\, \Lambda'} f([\Pi']).  
\end{aligned} 
\end{equation}
In words, the value of $T_p(f)$ on a lattice $[\Lambda']$ is the sum of the values of $f$ on the classes $[\Pi']$ for all $p$-neighbors $\Pi' \sim_p \Lambda'$.

In particular, if $f$ is the characteristic function of $[\Lambda]$, then $(T_p f)([\Lambda'])$ is the number of $p$-neighbors of $\Lambda'$ isometric to $\Lambda$.  Therefore, the matrix of $T_p$ in the basis of characteristic functions is the adjacency matrix $[T_p]$ of the graph $G_p$.

\begin{rmk}
Already the idea of a vector space spanned by isometry classes of quadratic forms, and the connection between Hecke operators and theta series (see below), appears in work of Eichler \cite[Kapitel 4]{Eichler:book}.
\end{rmk}

The operators $T_p$ pairwise commute and are self-adjoint with respect to a natural inner product.  So there is a basis of simultaneous eigenvectors, called \defi{eigenforms}.  Since $G_p$ is $k$-regular, the vector $e=(1,\dots,1)^{\intercal}$ is always an eigenvector with eigenvalue equal to $k$.  

We now give a few examples to illustrate how the Hecke operators (counting $p$-neighbors) provide a rich vein of arithmetic data.

\begin{example} \label{exm:eigenv1}
Continuing \Cref{exm:collectingt2}, we find eigenvectors $e=\begin{pmatrix} 1 \\ 1 \end{pmatrix}$ and $f=\begin{pmatrix} 2 \\ -3 \end{pmatrix}$ spanning $M(\Orth(\Lambda))$ (written in the basis of characteristic functions).
For $p \neq 11$, we have $T_p(e)=(p+1)e$ and $T_p(f)=a_p f$ with   
\begin{equation} 
a_2=-2,\ a_3=-1,\ a_5=1, \dots .
\end{equation}
These coefficients belong to the series
\[ f(q) = q\prod_{n=1}^{\infty} (1-q^n)^2(1-q^{11n})^2 = \sum_{n=1}^{\infty} a_n q^n \]
where $f(q) \in S_2(\Gamma_0(11))$ is the unique cuspidal newform of level $11$ and weight $2$ with LMFDB \cite{LMFDB} label \href{http://www.lmfdb.org/ModularForm/GL2/Q/holomorphic/11/2/a/a/}{\textsf{11.2.a.a}}.  A final alternate expression, coming from the modularity of elliptic curves, says that $a_p=p+1-\#A(\F_p)$ where $A$ is the elliptic curve \href{http://www.lmfdb.org/EllipticCurve/Q/11/a/3}{\textsf{11a3}} defined by
\[ A \colon y^2+y=x^3+x^2. \]
\end{example}

\begin{example} \label{exm:eigenv2}
In beautiful work (and recalling \Cref{exm:unimodular}), Chenevier--Lannes \cite[\S 1.2, Theorem A]{CL} show that for $n=16$ and $d=1$, we have
\begin{equation} 
[T_p] = \#D_{16}^+(\F_p)+(1+p+p^2+p^3)\frac{1+p^{11}-\tau(p)}{691} \begin{pmatrix} -405 & 286 \\ 405 & -286 \end{pmatrix}  
\end{equation}
in the basis $[E_8 \boxplus E_8]$, $[D_{16}^+]$ and where $\#D_{16}^+(\F_p)=\#(E_8 \boxplus E_8)(\F_p)$ is an explicit polynomial in $p$ \eqref{eqn:Qnum} and
\[ \Delta(q)=q \prod_{n=1}^{\infty} (1-q^n)^{24} = \sum_{n=1}^{\infty} \tau(n)q^n \]
is the discriminant modular form \href{http://www.lmfdb.org/ModularForm/GL2/Q/holomorphic/1/12/a/a/}{\textsf{1.12.a.a}} of weight $12$ and level $1$.  

Concisely put, the number of $p$-neighbors of $D_{16}^+$ isometric to $D_{16}^+$ is given by the coefficient $\tau(p)$, up to an explicit factor! 
\end{example}

Nebe--Venkov \cite{NV} begin their detailed investigations by computing the matrix $[T_2]$ for the Niemeier lattices, using Ph.D.\ thesis work of Borcherds \cite{Borcherds-thesis}.  The more general (and truly miraculous) story is told in the work of Chenevier--Lannes \cite[Theorem D]{CL}.  See also work of M\'egarban\'e \cite{Megarbane}, who studied lattices of rank $n=23,25$ with $d=1$.  

Algorithms for neighbors (cf.\ \cref{sec:appl}) extend to allow the computation of Hecke operators, including higher weight.  For a recent overview, see Assaf--Fretwell--Ingalls--Logan--Secord--Voight \cite[\S 3]{AFILSV}.

\subsection*{Theta series, and low rank}

We now briefly indicate how orthogonal modular forms are related to classical and Siegel modular forms via theta series, as indicated in the preceding examples.  For further reference, see e.g.\ Cohen--Str\"omberg \cite[\S 1.6]{CohS} and Freitag \cite{Freitag}.

The simplest theta series is
\begin{equation} 
\theta_{\Lambda}(q) \colonequals \sum_{x \in \Z^n} q^{Q(x)} = \sum_{m=0}^{\infty} r_Q(m) q^{m} \in \Z[[q]] 
\end{equation}
where $r_Q(m) \colonequals \#\{x \in \Z^n : Q(x)=m\}$ counts the number of representations of $m \in \Z_{\geq 0}$ (by the quadratic form $Q$ attached to $\Lambda$).  The series $\theta_{\Lambda}(q)$ is not only a formal power series in the variable $q$.  In fact, letting $q = e^{2\pi i z}$ with $z \in \C$ in the upper half-plane (i.e., $\impart z >0$), Poisson summation shows that $\theta_{\Lambda}$ is a classical modular form.  More precisely, it has level $4d$, weight $n/2$, and character $\chi_{d^*}$ attached to the quadratic field of discriminant $d^* \colonequals 1,(-1)^{n/2} d$ according as $n$ is odd or even.

More generally, we ask for the number of representations of other quadratic forms by $Q$.  For $g \in \Z_{\geq 1}$, we define the \defi{theta series}
\begin{equation}
\theta^{(g)}(\Lambda)(\tau) \colonequals \sum_{A\in \Mat_{n,g}(\Z)}e^{\pi i \tr(A^{\intercal}[T] A \tau)},
\end{equation}
where the sum is over $n \times g$ matrices $A$
and $\tau$ is a variable in the Siegel upper half-plane 
\begin{equation} 
\mathcal{H}_g \colonequals \{\tau\in M_g(\C) : \tau^{\intercal} = \tau \text{ and } \impart(\tau)>0\}.
\end{equation}
Generalizing the case $g=1$, the series $\theta^{(g)}$ is a Siegel modular form.  More precisely, extending linearly, we obtain a map $M(\Lambda) \to M_{n/2}(\Gamma_0^{(g)}(4d),\chi_{d^*})$ from the space of orthogonal modular forms to the space of Siegel modular forms of weight $n/2$, level $\Gamma_0^{(g)}(d)$, and quadratic character $\chi_{d^*}$.  

For $p\nmid d$, results of Rallis \cite{Rallis} relate the action of $p^k$-neighbor operators on an eigenform $\phi$ with the action of Hecke operators at $p$ acting on the associated Siegel modular form $\theta^{(g)}(\phi)$, which in particular relate the Hecke eigenvalues.  The statement of Rallis \cite[Remark 4.4]{Rallis} is also explained by Chenevier--Lannes \cite[Corollary 7.1.3]{CL} in the unimodular case.  See also work of Assaf--Fretwell--Ingalls--Logan--Secord--Voight \cite{AFILSV} for further discussion, references, and examples.  

When the number of variables $n$ is small, owing to the exceptional isomorphisms in Lie theory, this yields a tight connection between orthogonal modular forms and classical, Hilbert, and Siegel modular forms.  
\begin{itemize}
\item For $n=3$, the association between orthogonal modular forms as presented here and classical modular forms was first exhibited by Birch \cite{birch}.  Preceding Birch, Ponomarev \cite{Ponomarev:ternary} (building upon work of Eichler for $n=4$, see the next bullet point) observed a relation between certain lattices in the similitude genus and theta series, which he viewed as ``purely arithmetic''.  The work of Birch was refined and generalized by Hein \cite{Hein} and Hein--Tornar{\'\i}a--Voight \cite{HeinTornariaVoight}.  
\item For $n=4$, there is far more to say than can fit into this survey!  In some sense, the connection goes back as far as the original work of Brandt and Eichler in the case of quaternary forms attached to quaternion algebras.  For a partial overview, see Voight \cite[Chapter 41]{Voight}; to give just a few further references, we mention work of Yoshida \cite{Yoshida}, Ponomarev \cite{Ponomarev}, and B\"ocherer--Schulze-Pillot \cite{BSP}.
\item The case $n=5$ is more recent: see Dummigan--Pacetti--Rama--Tornar{\'\i}a \cite{dprt} for the case $d$ squarefree, building on previous work of Ibukiyama \cite{Ibu19}.  
\end{itemize}
In all three cases, the associations can be understood as being furnished by Clifford algebras.

\subsection*{Algebraic modular forms}

As briefly indicated at the end of \cref{sec:lattices}, the study of neighbors extends from the orthogonal group more generally by the theory of \emph{algebraic modular forms} due to Gross \cite{Gross1}.  

Let $F$ be a number field with ring of integers $\Z_F$ and $\widehat{F} \colonequals \tprodprime{\frakp} \,F_\frakp$ the finite adele ring of $F$ and $\widehat{\Z}_F \colonequals \prod_\frakp \Z_{F,\frakp} < \widehat{F}$ the profinite completion of $\widehat{\Z}$, so that $\widehat{F} = \widehat{\Z}_F \otimes_{\Z_F} F$.  Let $G$ be a reductive algebraic group over $F$.  Suppose that the real Lie group $G_\infty \colonequals G(F \otimes \R)$ is \emph{compact}.  Let $\widehat{K} < G(\widehat{F})$ be an open compact subgroup.  Then the double coset $G(F) \backslash G(\widehat{F}) / \widehat{K}$ is a finite set, an example of a zero-dimensional \emph{Shimura variety}.

\begin{example}
We may take $G \colonequals \Orth(V)$ for a positive definite quadratic space $V$ over a totally real field $F$: then $G_\infty = \prod_v \Orth(V_v)$ is the product over the real embeddings $v \colon F \hookrightarrow V$ of (compact) orthogonal groups.  For a $\Z_F$-lattice $\Lambda \subset V$ with $\widehat{\Lambda} \colonequals \Lambda \otimes \widehat{\Z}_F$, we may take 
\[ \widehat{K} \colonequals \Orth(\widehat{\Lambda})=\prod_\frakp \Orth(\Lambda_\frakp). \]
Then $G(\widehat{F})/\widehat{K}$ is in bijection with $\Gen \Lambda$, and the double coset $G(F) \backslash G(\widehat{F})/\widehat{K}$ is in bijection with $\Cl \Lambda$ \cite[\S 3]{GV}.
\end{example}

Let $\rho \colon G \to W$ be an algebraic (e.g., irreducible) representation, where $W$ is a finite-dimensional $E$-vector space with $[E:\Q]<\infty$.  An \defi{algebraic modular form} for $G$ of weight $W$ and level $\widehat{K}$ is a function $f \colon G(\widehat{F}) \to W$ such that
\[ f(\gamma \widehat{g}\widehat{u}) = \gamma f(\widehat{g}) \]
for all $\gamma \in G(F)$, $\widehat{g} \in G(\widehat{F})$, and $\widehat{u} \in \widehat{K}$.  Let $M_W(G,\widehat{K})$ be the $E$-vector space of algebraic modular forms for $G$ of weight $W$ and level $\widehat{K}$.  Write
\begin{equation} 
G(\widehat{F}) = \bigsqcup_{i=1}^h G(F) \widehat{x_i} \widehat{K} 
\end{equation}
and $\Gamma_i \colonequals G(\Q) \cap \widehat{x_i} \widehat{K} \widehat{x_i}^{-1}$ (a discrete subgroup of the compact group $G_\infty$, so finite).  Then we have an isomorphism
\begin{equation}
\begin{aligned}
M_W(G,\widehat{K}) &\xrightarrow{\sim} \bigoplus_{i=1}^h H^0(\Gamma_i,W) \\
f &\mapsto (f(\widehat{x_i}))_i
\end{aligned}
\end{equation}
where $H^0(\Gamma_i,W) = \{w \in W : \gamma w = w \text{ for all $\gamma \in \Gamma_i$}\}$ is the fixed subspace.  

The \defi{Hecke algebra} is the ring of locally constant, compactly supported, $\widehat{K}$-bi-invariant functions on $G(\widehat{F})$; it is generated by characteristic functions $T_{\widehat{p}}$ of $\widehat{K} \widehat{p} \widehat{K}$ for $\widehat{p} \in \widehat{G}$: writing 
\[ \widehat{K} \widehat{p} \widehat{K} = \bigsqcup_j \widehat{p}_j \widehat{K} \]
(a finite disjoint union) we define
\begin{equation} 
(T_{\widehat{p}} f)(\widehat{g}) = \sum_j^{<\infty} f(\widehat{g} \widehat{p_j}). 
\end{equation}

In this way, we recover the definition of the Hecke operator $T_p$ from \cref{sec:modularforms} as the operator $T_{\widehat{p}}$, defined as follows: when $\Lambda_p= H_p \boxplus \Lambda_p'$, as in \Cref{laticpr}(ii) for $F=\Q$ and $W=\Q$ the trivial representation, we take the operator with component $\diag(1/p,p,1,\dots,1)$ at $p$ and $1$ otherwise.

This lattice interpretation of Hecke operators was elaborated upon by Greenberg--Voight \cite{GV} for orthogonal and unitary groups and for symplectic groups by Chisholm \cite{Chisholm} and Sch\"onnen\-beck \cite{Schonnen}.  Using invariant forms, it would be interesting to extend this to include exceptional groups, building on the group-theoretic description due to Lansky--Pollack \cite{LP}.  Indeed, recent work by Cohen--Nebe--Plesken \cite{CNP} and Kirschmer \cite{Kirschmer-exceptional} use Kneser neighbors algorithmically to enumerate forms and one-class genera for exceptional groups.  And so we anticipate further developments in---and applications of---Kneser's theory of $p$-neighbors in the years to come!

\end{document}